\documentclass{amsart}

\usepackage{amsmath}
\usepackage{mathtools}
\usepackage{amsthm,amssymb,color,comment}
\usepackage{bbm}
\usepackage{hyperref}
\usepackage[TS1,T1]{fontenc}
\usepackage{inputenc}
\usepackage{dsfont}
\usepackage{tikz}
\usepackage{enumitem}
\usepackage{soul}
\usepackage[sort]{cite}

\numberwithin{equation}{section}

\newtheorem{thm}{Theorem}[section]
\newtheorem{prop}[thm]{Proposition}
\newtheorem{lem}[thm]{Lemma}
\newtheorem{cor}[thm]{Corollary}

\theoremstyle{definition}
\newtheorem{Ass}[thm]{Assumption}
\newtheorem{rem}[thm]{Remark}

\DeclareMathOperator{\dist}{dist}

\DeclareMathOperator{\DIV}{div}

\newcommand{\R}{\mathbb{R}}
\newcommand{\N}{\mathbb{N}}
\newcommand{\p}{\partial}
\newcommand{\eps}{\varepsilon}

\newcommand{\supp}{\text{supp}}
\newcommand{\diff}{\mathop{}\!\mathrm{d}}

\newcommand{\curl}{\mathrm{curl}}
\newcommand{\grad}{\nabla}

\makeatletter
\newcommand{\doublewidetilde}[1]{{%
  \mathpalette\double@widetilde{#1}%
}}
\newcommand{\double@widetilde}[2]{%
  \sbox\z@{$\m@th#1\widetilde{#2}$}%
  \ht\z@=.9\ht\z@
  \widetilde{\box\z@}%
}
\makeatother
\author{Christian Seis}
\address{Institut f\"ur Analysis und Numerik, Universit\"at M\"unster, Einsteinstr.~62, 48149 M\"unster, Germany}
\email{seis@uni-muenster.de}

\author{Emil Wiedemann}
\address{Department of Mathematics, Friedrich-Alexander-Universit\"at Erlangen-N\"urnberg, Cauerstr.~11, 91058 Erlangen, Germany}
\email{emil.wiedemann@fau.de}

\author{Jakub Woźnicki}
\address{Faculty of Mathematics, Informatics and Mechanics, University of Warsaw, Stefana Banacha 2, 02-097 Warsaw, Poland; Institute of Mathematics of Polish Academy of Sciences, Jana i J\k edrzeja \'Sniadeckich 8, 00-656 Warsaw, Poland}
\email{jw.woznicki@student.uw.edu.pl}

\begin{document}

\title[The 2D Viscosity Limit on a Bounded Domain]{Strong Convergence of Vorticities in the 2D Viscosity Limit on a Bounded Domain}

\begin{abstract}
In the vanishing viscosity limit from the Navier-Stokes to Euler equations on domains with boundaries, a main difficulty comes from the mismatch of boundary conditions and, consequently, the possible formation of a boundary layer. Within a purely interior framework, Constantin and Vicol showed that the two-dimensional viscosity limit is justified for any arbitrary but finite time under the assumption that on each compactly contained subset of the domain, the enstrophies are bounded uniformly along the viscosity sequence. Within this framework, we upgrade to local strong convergence of the vorticities under a similar assumption on the $p$-enstrophies, $p>2$. The key novel idea is the analysis of the evolution of the weak convergence defect.  
\end{abstract}


\maketitle

\section{Introduction}
\noindent Let $\Omega\subset\R^2$ be a bounded domain with a smooth boundary. Consider a family $\{u^\nu\}_{\nu > 0}$ of solutions to the 2D homogeneous, incompressible Navier-Stokes system
\begin{equation}\label{sys:navier_stokes}
\begin{aligned}
 \partial_t u^\nu  + \DIV_x (u^\nu \otimes u^\nu ) + \nabla_x p^\nu  &= \nu\Delta_x u^\nu&&\quad\mbox{in }(0,T)\times \Omega, \\
    \DIV_x u^\nu  &= 0&&\quad\mbox{in }(0,T)\times \Omega,
    \end{aligned}
\tag{NS}
\end{equation}
When this system is equipped with the Dirichlet (no-slip) boundary condition $u(x,t)=0$ for all $x\in\partial\Omega$ and $t>0$, it is classically known (see for instance~\cite[Remark 5.7]{galdinotes}) that there exists a unique smooth solution globally in time provided $u^\nu(0,\cdot)\in L^2(\Omega)$ and is divergence-free. 

Taking the ${\mathrm{curl}}$ of both sides of the equation \eqref{sys:navier_stokes}, we obtain the vorticity formulation
\begin{equation}\label{sys:navier_stokes_vorticity_formulation}
\begin{aligned}
\partial_t \omega^\nu + u^\nu\cdot\nabla\omega^\nu = \nu\Delta\omega^\nu\quad\mbox{in }(0,T)\times \Omega,\\
    \omega^\nu = \mathrm{curl}(u^\nu)\quad\mbox{in }(0,T)\times \Omega.
    \end{aligned}
\tag{NSV}
\end{equation}
As $u^\nu$ can be recovered uniquely from $\omega^\nu$ under the Dirichlet boundary condition, the vorticity formulation gives a complete description of the evolution. The parabolic character, however, is impaired as the boundary condition is imposed on $u^\nu$ and no apparent boundary information is available for the vorticity $\omega^\nu$. We will get back to this fundamental issue shortly. 

The subject of this paper is to study the vanishing viscosity limit, i.e., the limit $\nu\to0$. Formally, one expects convergence to the Euler equations, which are given by setting $\nu=0$ in~\eqref{sys:navier_stokes}:
\begin{equation}\label{sys:euler}
\begin{aligned}
\partial_t u + \DIV (u\otimes u) + \nabla p = 0\quad\mbox{in }(0,T)\times \Omega,\\
    \DIV u = 0\quad\mbox{in }(0,T)\times \Omega,
\end{aligned}\tag{E}
\end{equation}
or, in vorticity formulation,
\begin{equation}\label{sys:euler_vorticity_formulation}
\begin{aligned}
   \partial_t \omega + u\cdot\nabla\omega = 0\quad\mbox{in }(0,T)\times \Omega,\\
    \omega = \mathrm{curl}(u)\quad\mbox{in }(0,T)\times \Omega.
\end{aligned}\tag{EV}
\end{equation}
As the order of the system has decreased after setting $\nu=0$, the boundary condition changes and is usually chosen as $u(x,t)\cdot n(x)=0$ for all $x\in\partial\Omega$ and $t\geq0$, where $n(x)$ denotes the outer unit normal at $x\in\partial\Omega$. This condition is encoded in requiring $u(t)\in L^2(\Omega)$ and $\DIV u = 0$.

We impose two key assumptions in this paper. The first concerns the initial data. On the one hand, we require that the initial kinetic energy remain finite, regardless of the viscosity, which is a physically reasonable condition. On the other hand, we assume that the initial vorticities converge locally strongly in the inviscid limit. This assumption is both natural and necessary for the objectives of the present study.
\begin{Ass}\label{ass:data}
    Let $p > 1$ be arbitrary. For the initial data, we impose $\DIV u^\nu_0 = 0$,
    \begin{equation*}
        \sup_{\nu>0}\|u^\nu_0\|_{L^2(\Omega)}<\infty 
    \end{equation*}
    and, setting $\omega^\nu:=\curl (u^\nu)$ and $\omega:=\curl(u)$,
    \begin{equation*}
        \omega^\nu_0\to\omega_0\quad\text{in $L^p_{loc}(\Omega)$}.
    \end{equation*}
\end{Ass}

Our second assumption concerns the generation of vorticity in the interior of the domain. We assume that it can be controlled over a finite time span, uniformly with respect to viscosity.

\begin{Ass}\label{ass:assumption_on_linfty_bounds_vorticity}
Let $p > 1$ be arbitrary. For every compact set $K\subset \Omega$, there exists a constant $C_K$ such that 
$$
\sup_{\nu>0}\sup_{t\in (0, T)}\int_{K}|\omega^\nu(t,x)|^p \diff x \leq 
C_K.
$$
\end{Ass}

The latter condition appeared first in a study by Constantin and Vicol \cite{CoVi18} and also appeared in \cite{CLNV19}. Before discussing its justification and consequences, we state our main result.

\begin{thm}\label{mainthm}
Let $\Omega\subset\R^2$ be a smooth bounded domain, and $0<T<\infty$.       
Under Assumptions~\ref{ass:data} and~\ref{ass:assumption_on_linfty_bounds_vorticity} restricted further to only values $2 < p < +\infty$, $\omega^\nu$ defined as a strong solution to \eqref{sys:navier_stokes_vorticity_formulation}, converges, up to a subsequence as $\nu\to0$, to a weak solution $\omega$ of the Euler equations \eqref{sys:euler_vorticity_formulation} strongly in $C_{loc}([0,T);L^q_{loc}(\Omega))$ for any $1\leq q < p$. In particular, it holds that
\begin{equation}\label{XX}
\lim_{\nu\to 0} \nu\int_0^{T'} \int_K |\nabla \omega^{\nu}|^2\, \diff x \diff t=0,
\end{equation}
for any compact $K\subset \Omega$ and $T'<T$.
\end{thm}
Here, by convergence in $C_{loc}([0,T);L^q_{loc}(\Omega))$ we mean that for any $\chi\in C_c^\infty(\Omega)$ and $0 < T' < T$, we have $\chi\omega^\nu\to\chi\omega$ in $C([0,T'];L^q(\Omega))$. 

Let us give some context and explanation. The viscosity limit problem has been under intense mathematical scrutiny for decades. When there are no physical boundaries (that is, $\Omega=\R^d$ or $\Omega=\mathbb{T}^d$, the flat $d$-dimensional torus) and the solution of Euler is sufficiently smooth, then the convergence can easily be obtained by a relative energy estimate as given, for instance, in~\cite[Section 4.4]{PLL96}. Other methods and stronger results for the regular case without boundary can be found in~\cite{Swann,Kato72,Const86,Mas07}.

At low regularity, when no `good' solution of Euler is available, virtually nothing is known about the viscosity limit. (What counts as `good' depends on the space dimension, but in any case essentially one space derivative of velocity is needed.) But even when there is a strong solution of Euler, things get vastly more complicated as soon as a physical boundary enters the scene. The reason for this is the mismatch between the boundary conditions: For the Navier--Stokes equations, the natural boundary condition is no-slip, which requires the fluid to remain at rest on the boundary. As a result, tangential flows near the boundary generate vortices that could eventually detach and move toward the interior. In contrast, the Euler equations are governed by no-penetration boundary conditions, which allow for non-zero tangential flows.

A famous criterion of Kato~\cite{Kato84} states that the transition from Navier--Stokes to Euler is valid if and only if
\begin{equation*}
   \lim_{\nu\to0} \nu\int_0^T\int_{\Gamma_\nu}|\nabla u^\nu(x,t)|^2\diff x\diff t=0,
\end{equation*}
where $\Gamma_\nu:=\{x\in\Omega: \dist(x,\partial\Omega)<\nu\}$ is a region near the boundary, which is  much thinner than $\sqrt\nu$, as would be the thickness of the boundary layer expected from parabolic scaling. Sophisticated research has aimed at identifying sufficient conditions on the data and the regularity of the boundary under which the Kato criterion can be guaranteed to hold at least for short time, see for instance~\cite{SamCaf1, SamCaf2, Mae14, BarNgu22}.

Another line of thought was introduced by Constantin and Vicol~\cite{CoVi18}, who studied only the interior behaviour in the spirit of our Assumption \ref{ass:assumption_on_linfty_bounds_vorticity}. Before discussing this assumption, we try to make the reader sensitive for the challenges that occur in the presence of boundaries.

The vorticity formulations~\eqref{sys:navier_stokes_vorticity_formulation} and~\eqref{sys:euler_vorticity_formulation} are genuine to the two-dimensional situation (indeed, in higher dimensions an additional vortex stretching term appears). This vorticity transport structure greatly facilitates the analysis. For instance, in the case without boundaries, say $\Omega=\R^2$, we can multiply the vorticity equation~\eqref{sys:navier_stokes_vorticity_formulation} with $\omega^\nu$ itself to obtain, after integration by parts, the so-called enstrophy balance equation
\begin{equation}\label{enstrophybalance}
    \int_{\R^2}|\omega^\nu(t)|^2\diff x+2\nu\int_0^t\int_{\R^2}|\nabla\omega^\nu|^2\diff x\diff s=\int_{\R^2}|\omega_0^\nu|^2\diff x,
\end{equation}
which gives rise to a uniform a priori bound on $\|\omega^\nu\|_{L^2}$ (and more generally on the $L^p$-norms for any $1\leq p\leq\infty$). Regarding the viscosity limit, then, using such a priori bounds and other transport techniques such as DiPerna-Lions theory~\cite{diperna1989ordinary}, one can show many things provided the initial vorticity is in $L^p$: convergence of a vanishing viscosity subsequence to a weak solution of Euler~\cite{DiPM87}, renormalization of the viscosity limit~\cite{CriSpi15,CCNS17}, and energy conservation~\cite{CheLop16}. 

In fact, the usual convergence proof via weak compactness, as in~\cite{DiPM87}, only gives weak convergence of the vorticities. In the absence of boundaries, this has recently been improved to strong convergence in~\cite{CCS21,NSW21,CDE22}. 

In the case with physical boundaries,~\eqref{enstrophybalance} will not generally be true any longer: Indeed, the dissipative term multiplied by $\omega^\nu$ and integrated in space becomes
\begin{equation*}
    \nu\int_\Omega \omega^\nu\Delta\omega^\nu\diff x=-\nu\int_\Omega |\nabla \omega^\nu|^2\diff x+\nu\int_{\partial\Omega}\omega^\nu n\cdot\nabla\omega^\nu\diff S,
\end{equation*}
and the surface integral need not vanish, as we know nothing about the vorticity on the boundary. It is this integral that gives precise mathematical meaning to the expression \emph{vorticity production at the boundary}. Consequently, we lose the a priori estimate~\eqref{enstrophybalance}, and any attempt to establish the viscosity limit with standard methods appears futile.

Indeed, if the vorticity produced at the boundary propagates into the bulk of the fluid, we cannot hope for the Euler equations to provide an accurate description of the flow. Constantin and Vicol~\cite{CoVi18} instead impose (the $L^2$ version of) Assumption~\ref{ass:assumption_on_linfty_bounds_vorticity} to prevent an uncontrolled accumulation of interior vorticity. It is interesting here to recall the study~\cite{Kelli2} of Kelliher, who proves that the validity of the vanishing viscosity limit is in fact equivalent to the blow-up of vorticity in any $L^p$ norm with $p>1$ along the sequence (unless the limiting Euler solution happens to have zero tangential component on the boundary). Assumption~\ref{ass:assumption_on_linfty_bounds_vorticity} imposes that this blow-up occur at the boundary of the domain, which is perfectly consistent with Kelliher's observation~\cite{Kelli1} of the necessary formation of a vortex sheet at the boundary in the viscosity limit.  

The question whether and to what quantitative extent any vorticity produced at the boundary can accumulate inside the bulk via a mechanism of boundary layer detachment is analytically still open, and neither are we aware of any numerical studies in this regard. Yet, the criterion~\ref{ass:assumption_on_linfty_bounds_vorticity} is in a certain sense complementary to the one of Kato, as it shifts the focus away from the boundary and towards the interior. Indeed,  Assumption~\ref{ass:assumption_on_linfty_bounds_vorticity} and our results on local $L^p$ convergence allow, in principle, for an uncontrollable behavior near the boundary: The constants $C_K$ could explode wildly as $K$ gets closer to $\partial\Omega$, and so can the rate of convergence in Theorem~\ref{mainthm} and Proposition~\ref{byproduct}. It is in this sense that the theory presented here, and previously in~\cite{CoVi18}, is purely \emph{interior}.

 Constantin and Vicol established strong convergence of the velocity fields and weak convergence of the vorticities. A similar result was then proved~\cite{CLNV19} under a $p=1$ version of Assumption~\ref{ass:assumption_on_linfty_bounds_vorticity} (where, in addition, concentrations have to be excluded in the interior). 
We summarize these results in a statement suitable for our purposes:

\begin{prop}\label{byproduct}
    Let $\Omega\subset\R^2$ be a smooth bounded domain, and $0<T<\infty$.
Under Assumptions~\ref{ass:data} and~\ref{ass:assumption_on_linfty_bounds_vorticity}, up to a subsequence, $u^\nu$, defined as a strong solution to \eqref{sys:navier_stokes}, converges strongly to a weak solution $u$ of the Euler equations \eqref{sys:euler} in $C([0,T];L^2_{loc}(\Omega))$.
\end{prop}

We give a self-contained proof for the readers' convenience. Similarly to~\cite{CoVi18}, it relies on a localized form of the Calder\'on-Zygmund estimate relating velocity and vorticity (Lemma~\ref{cor:bounds_on_velocity_by_sobolev_embed} below). 

Our main Theorem~\ref{mainthm} states that it is possible to upgrade to strong convergence of the vorticities, just as in the case without boundaries, as long as $p>2$. Having strong convergence instead of only weak convergence means that oscillations in the vorticity distribution are ruled out, and that certain higher order moments
\[
I^{\nu}_n(t) = \int (\omega^{\nu}(t,\cdot))^n \diff x
\]
(with $n<p$) will be conserved in the inviscid limit. Equivalently, strong convergence entails that there is no anomalous dissipation of vorticity, see \eqref{XX}, which rules out fine scale mixing of the vorticity field. These observations are consistent with the Miller-Robert mean-field theory of   2D turbulence \cite{Miller90,Robert91}, see   \cite{CDE22} for further discussions.

 Due to the mentioned difficulties, the techniques that are available for the full space \cite{CCS21,NSW21,CDE22} do not apply. Instead, we need a key novel idea, which we consider to be interesting on its own: We consider the weak convergence defect $m$ and show that $m$ is a subsolution to the limiting vorticity transport equation (Lemma~\ref{L2}). The latter implies through some ingredients from DiPerna--Lions theory that $m=0$ on $\Omega$ (Lemma~\ref{L3}), and, thus, this shows strong convergence. 

We finally note that the assumption $p>2$ in Theorem \ref{mainthm} ensures that the velocity field is bounded, which will be crucial for our analysis. See Lemma \ref{L3} for details.



\section*{Acknowledgements.} This work is partially  funded by the Deutsche Forschungsgemeinschaft (DFG, German Research Foundation) under Germany's Excellence Strategy EXC 2044–390685587, Mathematics M\"unster: Dynamics--Geometry--Structure. 

Emil Wiedemann acknowledges support through DFG project no.~\mbox{525716336} within the DFG Priority Programme SPP 2410 ``Hyperbolic Balance Laws in Fluid Mechanics: Complexity, Scales, Randomness''. 

Jakub Woźnicki was supported by National Science Center, Poland through project no.\\ 2023/49/N/ST1/02737.

The authors would like to thank Felix Otto for valuable suggestions.

No data has been created or processed in this research.

\section{Preliminaries}\label{sec:prelim}
We recall that $\Omega\subset\R^2$ denotes a smooth bounded domain. In what follows, $x\in\Omega$ will always denote a spatial variable, while $t\in [0, T]$ is reserved for a time variable. For vector manipulations, we write $a\cdot b$ for a standard scalar product, whenever $a, b\in \R^d$. Similarly, by $A : B$, we denote the usual scalar product $\mathrm{tr}(A^T\, B)$ between two matrices $A, B\in \R^{d \times d}$. Throughout the paper, we shall denote by $p'$ a H\"{o}lder conjugate of an exponent $p\in [1, +\infty]$. We also remind the reader that by $\omega$ we will mean $\curl(u)$, which for $u\in C^1(\Omega; \R^2)$ reduces to 
$$
\omega := \curl(u) = \partial_{x_1}u_2 - \partial_{x_2}u_1.
$$
In this article, the operators $\nabla$, $\DIV$, and $\curl$ only act on the spatial variables.

We use the standard notation for Sobolev and Lebesgue functions and frequently do not distinguish between scalar- or vector-valued functions. With $C([0,T];L_{w}^p(\Omega))$ we denote the space of functions that are continuous from time into $L^p(\Omega)$ equipped with the weak topology. More precisely, we request the map 
\begin{equation*}
    [0,T]\to\R,\quad t\mapsto\int_\Omega f(t,x)\chi(x)\diff x
\end{equation*}
be continuous for every $\chi\in L^{p'}(\Omega)$.

For the two-dimensional Navier-Stokes equations~\eqref{sys:navier_stokes}, for any divergence-free initial data in $L^2(\Omega)$, there exists a unique solution that is smooth for all positive times. For the Euler equations, smoothness is not expected, and we need to define a weak notion of solution. For $0<T\leq\infty$, we call a divergence-free function $u\in L^\infty(0,T;L^2(\Omega))$ a weak solution of the Euler equations \eqref{sys:euler} with initial data $u_0$ if 
\begin{equation*}
    \int_0^T\int_\Omega \partial_t\psi\cdot u+\nabla\psi:u\otimes u \diff x\diff t +\int_\Omega u_0(x)\cdot\psi(0,x)\diff x =0
\end{equation*}
for all divergence-free vector fields $\psi\in C_c^\infty([0,T)\times\Omega)$. See~\cite{Wi2018} for a discussion of this notion of weak solution, particularly with respect to boundary effects. 

Taking the curl operator of the Euler equations, one arrives at the scalar vorticity equation~\eqref{sys:euler_vorticity_formulation}, whose weak form is
\begin{equation*}
    \int_0^T\int_\Omega \omega(\partial_t\psi+u\cdot \nabla\psi)\diff x\diff t+\int_\Omega \omega_0(x)\psi(0,x)\diff x=0
\end{equation*}
for every scalar $\psi\in C_c^\infty([0,T)\times\Omega)$.

We recall some notions of convergence: We say $f_n\to f$ in $C([0,T];L_{loc}^p(\Omega))$ if
\begin{equation*}
    \sup_{t\in[0,T]}\int_\Omega|f_n(t,x)-f(t,x)|^p\chi(x)\diff x\to0
\end{equation*}
for every $\chi\in C_c^\infty(\Omega)$. Note in particular that $f\in C([0,T];L_{loc}^p(\Omega))$ is defined, as an $L^p_{loc}(\Omega)$ function, for every (and not merely almost every) time.

At last, we introduce the standard mollification procedure. Let $\eta:\R^d \to \R$ be a standard regularizing kernel, i.e. $\eta$ is a smooth, non-negative function, compactly supported in a ball of radius one and fulfills $\int_{\mathbb{R}^d} \eta(x) \diff x = 1$. Then, we set $\eta_{\varepsilon}(x) = \frac{1}{\varepsilon^d} \eta\left(\frac{x}{\varepsilon}\right)$ and for arbitrary $f: \R^d \times [0,T] \to \mathbb{R}$, we define
$$
f^{\varepsilon}(t,x) = \int_{\R^d} \eta_{\varepsilon}(x-y) f(t,y) \diff y.
$$



\section{A Priori Estimates}
We begin our analysis with a very well-known property of the solutions to \eqref{sys:navier_stokes}.
\begin{lem}\label{lem:basic_bounds_for_velocity}
Let $\{u^\nu\}_{\nu>0}$ be a family of solutions to the \eqref{sys:navier_stokes} whose initial data is divergence-free and bounded in $L^2(\Omega)$, uniformly in $\nu>0$. Then, $\{u^\nu\}_{\nu >0}$ is uniformly bounded in $L^\infty(0,T;L^2(\Omega))$.
\end{lem}
\begin{proof}
Multiply the equation \eqref{sys:navier_stokes} by $u^\nu$ and integrate by parts to obtain the bound, using the Dirichlet boundary condition.
\end{proof}
\begin{lem}\label{cor:bounds_on_velocity_by_sobolev_embed}
    For $1<p<\infty$, let $\{u^\nu\}_{\nu>0}$ be a sequence of strong solutions to the \eqref{sys:navier_stokes} and $\{\omega^\nu\}_{\nu > 0} = \{\mathrm{curl}(u^\nu)\}_{\nu > 0}$ satisfy Assumptions~\ref{ass:data} and~\ref{ass:assumption_on_linfty_bounds_vorticity}. Then, for any $\chi\in C^\infty_c(\Omega)$, $\{\chi\,u^\nu\}_{\nu > 0}$ is uniformly bounded in $L^\infty(0,T; W^{1,p}_{0}(\Omega))$. 
        In particular, if $p>2$,  $\{\chi\,u^\nu\}_{\nu > 0}$ is uniformly bounded in $L^\infty(0,T; C(\overline{\Omega}))$.
\end{lem}
\begin{proof}
We first discuss the case $p\geq 2$, as $p< 2$ is handled in a similar but slightly simpler way.

Consider the Dirichlet problem
\begin{equation}\label{sys:artificial_laplace}
    \begin{split}
        \left\{\begin{array}{ll} \Delta z^\nu = u^\nu\text{ in }\Omega\\
        z^\nu = 0 \text{ on }\p\Omega.
        \end{array}\right.
    \end{split}
\end{equation}
By~\cite[Theorem 9.15]{gilbarg2015elliptic} and Lemma~\ref{lem:basic_bounds_for_velocity}, this problem has a solution in $L^\infty(0,T;W^{2,2}(\Omega))$, and the norm in this function space is bounded uniformly in $\nu$. 

Fix $\chi\in C^\infty_c(\Omega)$ and compute, using \eqref{sys:artificial_laplace}, 
\begin{align}\label{artificial_laplace_curl}
\Delta\curl(\chi z^\nu) = \nabla^\perp\chi\cdot u^\nu+2\nabla\nabla^\perp\chi:\nabla z^\nu+\Delta \nabla^\perp\chi\cdot z^\nu+\Delta\chi\,\curl z^\nu+2\nabla\chi\cdot \nabla\curl z^\nu+\chi\omega^\nu.
\end{align}
Using Assumption \ref{ass:assumption_on_linfty_bounds_vorticity} and Lemma \ref{lem:basic_bounds_for_velocity} we know that the right-hand side of \eqref{artificial_laplace_curl} is bounded in $L^\infty(0,T;L^2(\Omega))$, uniformly in $\nu$. Hence, by~\cite[Theorem 9.15]{gilbarg2015elliptic}, $\curl(\chi z^\nu)$ is bounded in $L^\infty(0,T;W^{2,2}(\Omega))$. 

Similarly, taking now the divergence instead of the curl, one gets
\begin{align}\label{artificial_laplace_div}
   \Delta(\DIV(\chi z^\nu)) = \Delta\nabla\chi\cdot z^\nu+\nabla\chi\cdot u^\nu+2\nabla\nabla\chi:\nabla z^\nu+\Delta\chi\,\DIV z^\nu +\chi\DIV u^\nu+2\nabla\chi\cdot\nabla\DIV z^\nu.
\end{align}
Again, the right-hand side of \eqref{artificial_laplace_div} is bounded in $L^\infty(0,T;L^2(\Omega))$, independently of $\nu$, meaning that $\DIV(\chi z^\nu)$ is bounded in $L^\infty(0,T;W^{2,2}(\Omega))$. (Note $\DIV u^\nu=0$.) Using the identity
\begin{align}\label{artificial_laplace_identity}
\Delta (\chi z^\nu) = \nabla\DIV(\chi z^\nu) + \nabla^{\bot}\curl(\chi z^\nu),
\end{align}
we deduce that $\chi\, z^\nu$ is bounded in $L^\infty(0, T; W^{3,2}(\Omega))$. Therefore,
\begin{align}\label{yetanotheridentity}
\chi u^\nu=\chi\Delta z^\nu=\Delta(\chi z^\nu)-\Delta\chi z^\nu-2\nabla\chi\cdot\nabla z^\nu
\end{align}
is bounded in $L^\infty(0,T;W^{1,2}(\Omega))$ and therefore (by virtue of the embedding $W^{1,2}\subset L^p$) also in $L^\infty(0,T;L^p(\Omega))$, always uniformly in $\nu$.

Owing to Assumption \ref{ass:assumption_on_linfty_bounds_vorticity}, because the argumentation so far applies to any choice of test function, a second glance at~\eqref{artificial_laplace_curl} reveals that $\curl(\chi z^\nu)$ is bounded in $L^\infty(0,T;W^{2,p}(\Omega))$, because the right hand side of~\eqref{artificial_laplace_curl} is, in fact, bounded in $L^\infty(0,T;L^p(\Omega))$.
Similarly, a re-examination of~\eqref{artificial_laplace_div} informs us that $\DIV (\chi z^\nu)$ is also bounded in $L^\infty(0,T;W^{2,p}(\Omega))$, whence from~\eqref{artificial_laplace_identity} we discover $\Delta(\chi z^\nu)$ to be bounded in $L^\infty(0,T;W^{1,p}(\Omega))$.
Therefore, the right-hand side of~\eqref{yetanotheridentity} admits a bound in $L^\infty(0,T;W^{1,p}(\Omega))$, uniform in viscosity, which completes the proof of the first statement.

     The second statement then follows by Morrey's inequality (as we work in 2D).

     As for the case $p < 2$, one can notice that by H\"{o}lder's inequality and Lemma \ref{lem:basic_bounds_for_velocity}, $\{u^\nu\}_{\nu > 0}$ is bounded in $L^\infty(0, T; L^p(\Omega))$, thus we can repeat the argument going from \eqref{sys:artificial_laplace}, but by deducing that $\{z^\nu\}_{\nu > 0}$ is bounded in $L^\infty(0, T; W^{2,p}(\Omega))$, and using the local $L^\infty(0, T; L^p(\Omega))$ bounds of $\{u^\nu\}_{\nu > 0}$ and $\{\omega^\nu\}_{\nu > 0}$.

\end{proof}

We now show that to leading order in $\nu$, the solution to the Navier-Stokes equation is a renormalized subsolution to the transport equation.

\begin{lem}\label{lem:needed_bounds_for_vorticity_p}
    Let $2<p<\infty$ and $1\leq q\leq p$. Consider the sequence $\{\omega^\nu\}_{\nu > 0}$ of solutions to \eqref{sys:navier_stokes_vorticity_formulation} satisfying Assumption \ref{ass:assumption_on_linfty_bounds_vorticity}. Then, for any nonnegative  function $\psi\in C^\infty_c([0,T)\times \Omega)$, it holds that
      \begin{equation}
      \label{4}
   0 \le \int_0^T\int_\Omega |\omega^{\nu}|^q \left(\partial_t \psi  + u^\nu\cdot \grad\psi\right)\, \diff x \diff t +\int_\Omega |\omega_0^{\nu}|^q \psi(0,\cdot)\, \diff x + O(T \nu).
      \end{equation}
    
\end{lem}
\begin{proof}
We multiply \eqref{sys:navier_stokes_vorticity_formulation} by $  |\omega^\nu|^{q-2}\,\omega^\nu\,\psi$ and integrate the equality by parts to obtain (after a  straightforward computation which is justified thanks to the smoothness of $\omega^\nu$)
    \begin{align}\label{renormnu}
      \frac{d}{dt} \int_\Omega |\omega^{\nu}|^q \psi\, \diff x +  \nu q (q-1) \int_\Omega |\omega^{\nu}|^{q-2}|\grad\omega^{\nu}|^2\psi\, \diff x  =   \int_\Omega |\omega^{\nu}|^q \left(\partial_t \psi + u^{\nu}\cdot \grad \psi    + \nu \Delta \psi\right)\,\, \diff x.
    \end{align}
    Thanks to Lemma \ref{cor:bounds_on_velocity_by_sobolev_embed} (which ensures integrability of the advective term) and Assumption \ref{ass:assumption_on_linfty_bounds_vorticity}, the viscosity-dependent term on the right-hand side is uniformly controlled for any $q\le p$: Indeed, if $K\subset\Omega$ is compact and contains the support of $\psi$, then
    \[
   \left| \nu \int_\Omega |\omega^{\nu}|^q \Delta \psi\, dx\right| \lesssim \nu \|\omega^{\nu}\|_{L^q(K)}^q \lesssim \nu \|\omega^{\nu}\|_{L^p(K)}^q \le \nu C_K^{q/p},
    \]
     and thus, using the nonnegativity of the dissipation term on the left-hand side and integrating in time, we conclude the  statement of the lemma.
\end{proof}

\section{Strong convergence of the velocity field}


The present section is devoted to proving the following lemma concerning the strong convergence of the velocity field.
\begin{lem}\label{lem:strong_conv_of_velocity}
    Let $1<p<\infty$ and $\{u^\nu\}_{\nu>0}$ be a sequence of solutions to the \eqref{sys:navier_stokes} and $\{\omega^\nu\}_{\nu > 0} = \{\mathrm{curl}(u^\nu)\}_{\nu > 0}$ satisfy Assumptions~\ref{ass:data} and~\ref{ass:assumption_on_linfty_bounds_vorticity}. Then, there exists \mbox{$u\in L^\infty(0, T; L^2(\Omega))$} such that (up to a subsequence)
    \begin{enumerate}[label=(S\arabic*)]
        \item $u^\nu \stackrel{*}{\rightharpoonup} u$ weakly* in $L^\infty(0, T; L^2(\Omega))$\label{conv:weak_star_l2_conv_of_velocity},
        \item $\chi \nabla u^\nu \stackrel{*}{\rightharpoonup} \chi \nabla u$ weakly* in $L^\infty(0, T; L^{p}(\Omega))$\label{conv:weak_star_sobolev_loc_conv_of_velocity}, for any $\chi\in C^\infty_c(\Omega)$,
        \item $\chi\,u^\nu \rightarrow \chi\,u$ strongly in $C([0, T];L^{2}(\Omega))$, for any $\chi\in C^\infty_c(\Omega)$\label{conv:strong_conv_of_velocity}.
            \end{enumerate}
            Moreover, if $p>2$, then $\chi\,u^\nu \rightarrow \chi\,u$ strongly in $C([0, T]\times \overline{\Omega})$, for any $\chi\in C^\infty_c(\Omega)$.
\end{lem}
\begin{proof}
    The \ref{conv:weak_star_l2_conv_of_velocity} and \ref{conv:weak_star_sobolev_loc_conv_of_velocity} convergences follow immediately from the Banach-Alaoglu theorem and the bounds from Lemma \ref{lem:basic_bounds_for_velocity} and Lemma \ref{cor:bounds_on_velocity_by_sobolev_embed}. For~\ref{conv:weak_star_sobolev_loc_conv_of_velocity}, one should note additionally that it suffices to consider a countable set of cutoff functions $\chi$, so that the subsequence may be selected via a diagonal argument.
    
    To obtain \ref{conv:strong_conv_of_velocity}, fix an arbitrary Lipschitz subdomain $\Omega'\subset\Omega$ such that $\overline{\Omega'}\subset\Omega$, and define
    $$
    W^{1,p}_{0, \sigma}(\Omega') := \text{closure of $\{v\in C^\infty_c(\Omega'),\, \DIV v = 0\}$ in the $W^{1,p}(\Omega')$ norm}.
    $$  
    Fix $\psi\in L^1(0, T; W^{1,p'}_{0,\sigma}(\Omega'))$, multiply \eqref{sys:navier_stokes} by it, and integrate by parts to obtain
    \begin{align}
        \int_0^T\int_{\Omega'}\p_t u^\nu\cdot\psi\diff x\diff t = \int_0^T\int_{\Omega'}u^\nu\otimes u^\nu : \nabla\psi\diff x\diff t - \nu\int_0^T\int_{\Omega'}\nabla u^\nu : \nabla \psi\diff x\diff t.
    \end{align}
    Using H\"{o}lder's inequality, we get
    \begin{align*}
        \left|\int_0^T\int_{\Omega'}u^\nu\otimes u^\nu : \nabla\psi\diff x\diff t\right|&\lesssim \|u^\nu\|_{L^\infty((0, T); L^{2p}(\Omega'))}^2\|\psi\|_{L^1(0, T; W^{1,p'}_{0,\sigma}(\Omega'))},\\
        \left|\nu\int_0^T\int_{\Omega'}\nabla u^\nu : \nabla \psi\diff x\diff t\right| &\lesssim \nu\|u^\nu\|_{L^\infty(0, T; W^{1,p}(\Omega'))}^2\|\psi\|_{L^1(0, T; W^{1,p'}_{0,\sigma}(\Omega'))},
    \end{align*}
    which, by Lemma \ref{cor:bounds_on_velocity_by_sobolev_embed} and the embedding $W^{1,p}(\Omega')\subset L^{2p}(\Omega')$, implies that $\{\p_t u^\nu\}_{\nu > 0}$ is bounded in $L^\infty(0, T; (W^{1,p'}_{0,\sigma}(\Omega'))^*)$. Since $W^{1, p}(\Omega')$ embeds compactly into $L^2({\Omega'})$, which embeds continuously into $(W^{1,p'}_{0,\sigma}(\Omega'))^*$, we may use the Aubin-Lions Lemma \ref{aubin_lions_lemma} to deduce
    $$
    u^\nu \rightarrow u \text{ strongly in }C([0, T]; L^2({\Omega'})). 
    $$
    This shows \ref{conv:strong_conv_of_velocity}.

    The `moreover' part about uniform convergence follows in the same way, using the improved embedding $W^{1,p}(\Omega')\subset C(\overline{\Omega'})$ when $p>2$.
\end{proof}

\begin{rem}
In fact, using the weak convergence of  \ref{conv:weak_star_l2_conv_of_velocity}, we may pass to the inviscid limit in the weak formulation of the incompressibility condition in the Navier--Stokes problem, and we find that the limit vector field $u$ is itself divergence-free, even when extended trivially to all of $\R^2$. 
As a consequence,  $u$ automatically satisfies the boundary condition $u\cdot n=0$ on $\partial\Omega$ in a weak sense, cf.~\cite[Theorem III.2.3]{galdi2011introduction}.
\end{rem}

{\em Proof of Proposition~\ref{byproduct}.} 
From $u^\nu\to u$ in $C([0,T];L^2_{loc}(\Omega))$ and $u\in L^\infty(0,T;L^2(\Omega))$ is divergence-free, it follows almost immediately that $u$ is a weak solution. Indeed, let $\psi\in C_c^\infty([0,T)\times\Omega)$ be divergence-free and choose $\chi\in C_c^\infty(\Omega)$, such that $\chi=1$ on the spatial support of $\psi$. Then, on said support, $u^\nu\to u$ strongly in $L^2$. This allows to pass to the limit $\nu\to 0$ in all integrals appearing in the weak formulation, including the nonlinear one $\int_0^T\int_\Omega \nabla\psi:u^\nu\otimes u^\nu\diff x\diff t$.

\qed

\section{Weak convergence of the vorticity and the renormalization of the limiting equation}

\begin{lem}\label{L1}
    Let $p > 2$. Consider the sequence $\{\omega^\nu\}_{\nu > 0}$ of solutions to \eqref{sys:navier_stokes_vorticity_formulation} satisfying Assumptions~\ref{ass:data} and~\ref{ass:assumption_on_linfty_bounds_vorticity}. Then, there exists $\omega\in L^\infty(0, T; L^p_{loc}(\Omega))$ such that (up to a subsequence)
    \begin{align}\label{conv:weak_conv_vorticity}
    \chi\omega^\nu \to \chi\omega \quad\text{in $C([0, T]; L_w^p(\Omega))$ for each $\chi\in C_c^\infty(\Omega)$},
    \end{align}
    and it satisfies \eqref{sys:euler_vorticity_formulation} in the sense of distributions, with $u$ defined in Lemma \ref{lem:strong_conv_of_velocity}. Moreover, $\omega$ is renormalized, and thus, in particular,
        \begin{align}\label{eq:renormalized_weak_formulation_vorticity}
        \int_{0}^T\int_{\Omega}|\omega|^q\,\left(\p_t\psi +  u\cdot\nabla\psi\right)\diff x\diff t  + \int_{\Omega}|\omega_0|^q\psi(0, x)\diff x =0,
    \end{align}
    for any $1\leq q\leq p$ and $\psi\in C_c^{\infty}([0,T)\times \Omega)$.
\end{lem}

\begin{proof}
    The convergence $\chi\omega^\nu\to\chi\omega$ in the weak* topology of $L^\infty(0,T;L^p(\Omega))$ is already established in  \ref{conv:weak_star_sobolev_loc_conv_of_velocity}.
    For the weak continuity in time, notice
    \begin{equation*}
        \partial_t (\chi\omega^\nu)=\nu\chi\Delta\omega^\nu-\chi u^\nu\cdot\nabla\omega^\nu,
    \end{equation*}
    whence we will obtain an estimate for $\partial_t(\chi \omega^\nu)$ as follows:
Let $\phi\in W_0^{s,2}(\Omega)$, for some $s>0$ yet to be determined, then
    \begin{equation*}
        \left|\int_\Omega \partial_t(\chi\omega^\nu)\phi\diff x\right|\leq \nu\left|\int_\Omega \chi\Delta\omega^\nu\phi\diff x\right|+\left|\int_\Omega\chi\DIV(u^\nu\omega^\nu)\phi\diff x\right|=:\nu I_1+I_2.
    \end{equation*}
 Let $K:=\supp(\chi)$. Then on the one hand, keeping in mind Assumption~\ref{ass:assumption_on_linfty_bounds_vorticity},
    \begin{equation*}
        I_1=\left|\int_\Omega \omega^\nu\Delta(\chi\phi)\diff x\right|\leq C_K^{1/p}\|\chi\|_{C^2}\|\phi\|_{W^{2,p'}}.
    \end{equation*}
    On the other hand,
    \begin{equation*}
        I_2=\left|\int_\Omega \nabla(\chi\phi)\cdot u^\nu\omega^\nu\diff x\right|\leq C_K^{1/p}
    \|u^\nu\|_{L^{p'}(K)}    \|\chi\|_{C^1}\|\phi\|_{C^1}.
    \end{equation*}
Note that $\|u^\nu\|_{L^\infty}$ is uniformly (in $\nu$) bounded thanks to Lemma~\ref{cor:bounds_on_velocity_by_sobolev_embed}. Therefore, both $I_1$ and $I_2$ can be estimated by $\|\phi\|_{W^{s,2}}$ for sufficiently large $s$ (in fact, $s=2$ will do as $p>2$ and thus $p'<2$). As these estimates are uniform in time, we have thus showed  
    \begin{equation*}
        \sup_{0<\nu<1}\|\partial_t \omega^\nu\|_{L^\infty(0,T;W^{-s,2}(\Omega))}<\infty.
    \end{equation*}
    From the Aubin-Lions lemma~\ref{aubin_lions_lemma} with $X_0=L^p(\Omega)$, $X=X_1=W^{-s,2}(\Omega)$, we deduce $\chi\omega^\nu\to\chi\omega$ in $C([0,T];W^{-s,2}(\Omega))$. Since $(\chi\omega^\nu)$ is bounded in $L^\infty(0,T;L^p(\Omega))$ and $W_0^{s,2}$ is dense in $L^{p'}(\Omega)$, we even obtain convergence in $C([0,T];L^p_w(\Omega))$ as claimed.

    It remains to discuss the limiting equation. Passing to the limit in the distributional formulation of $\omega^{\nu} = \curl (u^{\nu})$ we see that $\omega = \curl (u)$.   To see that $\omega$ is in fact a distributional solution of \eqref{sys:euler_vorticity_formulation}, multiply \eqref{sys:navier_stokes_vorticity_formulation} by an arbitrary \mbox{$\psi\in C^\infty_c([0, T)\times \Omega)$} and integrate by parts to get
    \begin{align}\label{weak_formulation_vorticity_nu}
        \int_0^T\int_{\Omega}\omega^\nu\,\p_t\psi + (u^\nu\cdot\nabla\psi)\,\omega^\nu\diff x\diff t = -\nu\int_0^T\int_{\Omega}\omega^\nu\,\Delta\psi\diff x\diff t - \int_{\Omega}\omega_0^\nu\,\psi(0, x)\diff x.
    \end{align}
   Thanks to Lemma  \ref{cor:bounds_on_velocity_by_sobolev_embed} and the compact embedding $W^{1,p} \subset L^{p'}$ for $p>4/3$, we can easily pass to the limit in all of the terms in \eqref{weak_formulation_vorticity_nu} and get
    \begin{align}\label{weak_formulation_vorticity}
        \int_0^T\int_{\Omega}\omega\,\p_t\psi + (u\cdot\nabla\psi)\,\omega\diff x\diff t = -\int_{\Omega}\omega_0\,\psi(0, x)\diff x.
    \end{align}
    The renormalization property is a consequence of the DiPerna-Lions theory in \cite{diperna1989ordinary}, therefore we skip the argument here. As the proof contains some minor changes, we refer the interested reader to Appendix \ref{proof_renormalization}.
\end{proof}

\begin{cor}\label{ctycor}
Under the assumptions of Lemma~\ref{L1}, for any $\chi\in C_c^\infty(\Omega)$, 
\begin{equation*}
   \chi\omega\in C([0,T];L^p(\Omega)).
\end{equation*}

\end{cor}
\begin{proof}
The argument runs very similarly to~\cite[Corollary II.2]{diperna1989ordinary}. Indeed, ~\eqref{eq:renormalized_weak_formulation_vorticity} entails
\begin{equation}\label{renormlimit}
        \frac{\diff}{\diff t}\int_\Omega|\chi\omega(t)|^p\diff x=\int_\Omega |\omega(t)|^p u\cdot\nabla|\chi|^p\diff x,
\end{equation}
and the right hand side is $L^\infty(0,T)$. This shows continuity of $t\mapsto \|\chi\omega(t)\|_{L^p(\Omega)}$. By Lemma~\ref{L1}, $\chi\omega\in C([0,T];L^p_w(\Omega))$. But since weak convergence and convergence of the norms jointly imply strong convergence, we conclude. 


\end{proof}

\section{Strong convergence of vorticity}

In this section, we introduce the weak* vorticity defect as a key tool. In the terminology of~\cite[Section 11.1]{MajdaBertozzi}, it is the weak* defect measure associated with the sequence $\{|\omega^\nu|^q\}$, although we choose $q$ small enough so that the defect measure becomes in fact a function in $L^\infty(0,T;L_{loc}^{\frac{p}q}(\Omega))$. It turns out the defect satisfies a transport inequality, and in order to put this to use, we need to ensure the defect is a renormalised solution of said inequality, which necessitates the stated integrability and thus the requirement $q<p$.

\begin{lem}\label{L2}
   Let $p > 2$, and $1\leq q < p$. Consider the sequence $\{\omega^\nu\}_{\nu > 0}$ of solutions to \eqref{sys:navier_stokes_vorticity_formulation} satisfying Assumptions~\ref{ass:data}and~ \ref{ass:assumption_on_linfty_bounds_vorticity}. Then, there exists a function
   $$
        m\in L^\infty(0, T; L_{loc}^{\frac{p}q}(\Omega))
   $$
   such that, for each $\tilde\Omega$ compactly contained in $\Omega$,
   $$
        |\omega^\nu|^q - |\omega|^q \stackrel{*}{\rightharpoonup} m \quad\text{weakly* in $L^\infty(0, T; L^{\frac{p}{q}}(\tilde\Omega))$}
   $$
   and
   \begin{align}\label{1}
       \partial_t m + u\cdot\nabla m \leq 0
   \end{align}
   in the sense of distributions, that is:
   \begin{equation*}
       \int_0^T\int_\Omega (\partial_t\psi+u\cdot \nabla \psi)m\diff x\diff t\geq 0
   \end{equation*}
   for all nonnegative $\psi\in C_c^\infty([0,T)\times\Omega)$.
\end{lem}


\begin{proof}We combine the estimates in \eqref{4} and \eqref{eq:renormalized_weak_formulation_vorticity} to the effect that
\begin{align*}
0& \le \int_0^T\int_\Omega \left(|\omega^{\nu}|^q - |\omega|^q\right) \left(\partial_t \psi +u\cdot \grad\psi\right)\, \diff x\diff t\\
&\quad  +\int_\Omega \left(|\omega_0^{\nu}|^q-|\omega_0|^q\right)\psi\, \diff x  + \int_0^T\int_\Omega |\omega^{\nu}|^q\left(u^{\nu}-u\right)\cdot\grad \psi\, \diff x \diff t +O(\nu T)
\end{align*}
for any nonnegative $\psi \in C_c^{\infty}([0,T)\times \Omega)$. In view of the general Assumption \ref{ass:assumption_on_linfty_bounds_vorticity}, the difference $|\omega^{\nu}|^q-|\omega|^q$ is uniformly bounded in $L^{\infty}L^{\frac{p}q}$ on the support of $\psi$, and has thus a weakly* converging subsequence in that space. By arbitrary choice of $\psi$ and the usual diagonal argument, we can take the limit $m$ to be in $L^\infty(0,T;L^{\frac pq}_{loc}(\Omega))$. This settles the convergence of the first term on the right-hand side. 

The second term vanishes in the zero-viscosity limit due to Assumption~\ref{ass:data} on the initial data. The third one converges to zero by virtue of the uniform result in \ref{conv:strong_conv_of_velocity} and the overall Assumption \ref{ass:assumption_on_linfty_bounds_vorticity}. In the limit $\nu\to 0$, we thus arrive at
\[
0\le \int_0^T\int_\Omega m\left(\partial_t \psi + u\cdot \grad \psi\right)\, \diff x \diff t,
\]
as desired.
\end{proof}

\begin{rem}
    Note that \eqref{1} in fact tells us that the weak material derivative of $m$ is nonpositive. Therefore, moving into Lagrangian coordinates, one can expect, since $m$ starts from $0$, that it has to be nonpositive. Nevertheless, to prove this observation rigorously we must carefully craft a test function for our weak formulation. This is done in the lemma below.
\end{rem}

In what follows, we assume that the exponent $q$ in Lemma \ref{L2} satisfies $q\leq p-1$, so that $\frac pq \geq \frac{p}{p-1}$, the latter fraction being the dual exponent of $p$.

\begin{lem}\label{L3}
Suppose that $p> 2$ and $1\leq q\leq p-1$. Let $m\in L^{\infty}(0,T;L_{loc}^{\frac{p}{q}}(\Omega))$ be a distributional solution of \eqref{1} such that $m(0,\cdot)=0$. Then for any compact set $K$ there exists a time $T_K>0$ such that  $m(t,\cdot)\le 0$ in $K$ for almost every $t\in [0,T_K]$. 
\end{lem}

\begin{proof}
We choose  arbitrarily a compact set $K$ and an open set $U$ such that $K\subset \subset U\subset\subset \Omega$.

As $u\in L^\infty(0,T;W_{loc}^{1,p}(\Omega))$, by DiPerna-Lions theory~\cite[Theorem III.2]{diperna1989ordinary} the transport by $u$ can be represented by a Lagrangian flow $X(t,x)$. More precisely, the flow is given by 
\begin{align*}
    \partial_tX(t,x)=u(t,X(t,x)),\quad \quad X(0,x)=x.
\end{align*} 
As $u$ is continuous, there exists a number $T_K>0$, depending also on $U$, and another open set $U\subset\subset\tilde U\subset\subset\Omega$ such that $X^{-1}(t,U)\subset\tilde U$ for all $t \in [0, T_K]$.


Let now $\eta\in C^\infty_c((0,T_K)\times U)$ be nonnegative and consider the final-time problem for the forced transport equation 
\begin{equation}\label{2}
\partial_t \phi + u\cdot \grad \phi = - \eta,\quad \phi(T_K,\cdot) = 0.
\end{equation}

Using the method of characteristics, the solution to this problem is explicitly characterized by 
\begin{equation*}
    \phi(t,X(t,x))=\int_t^{T_K}\eta(s,X(s,x))\diff s,
\end{equation*}
whence $\phi$ is seen to be supported on $\tilde U$ for all times in $[0,T_K]$ (because $\supp (\eta)\subset U$ and $U\supset X^{-1}(t,U)$), and $\phi\geq 0$ on that same time interval.


At this point we are tempted to use $\phi$, extended trivially by zero past $T_K$, as a test function in~\eqref{1}. Alas, the transport equation propagates only low regularity (more precisely: a logarithm of a derivative \cite{CrippaDeLellis08,Jabin16,Leger18,AlbertiCrippaMazzucato19,
crippa2021growth,BrueNguyen21,MeyerSeis23}),  so that $\phi$ is not a valid test function in \eqref{1} without further regularization. Only the space regularity is an issue here, so we approximate $\phi$ by space convolution as in Lemma \ref{L1} to obtain
\begin{equation}
\label{3}
\partial_t\phi^{\eps} + u\cdot \grad\phi^{\eps} = -\eta^{\eps}+r^{\eps},\quad r^{\eps}= (u\cdot \grad \phi^{\eps}) - (u\cdot \grad\phi)^{\eps},
\end{equation}
and the commutator $r^{\eps}$      vanishes in $L^1(0,T_K;L^p_{loc}(\Omega))$ because $\phi\in L^{\infty}((0,T_K)\times\Omega)$, cf.~Lemma \ref{commutator_lemma} in the appendix. 
Using $\phi^{\eps}$  in \eqref{1}, we  now obtain
\[
\int_0^{T_K}\int_\Omega m\left(\partial_t \phi^{\eps} + u\cdot \grad\phi^{\eps}\right)\, \diff x \diff t \ge 0, 
\]
and then, by applying \eqref{3},
\[
\int_0^{T_K} \int_\Omega m \eta^{\eps}\, \diff x\diff t \le \int_0^{T_K}\int_\Omega m r^{\eps}\, \diff x \diff t.
\]
Because $m\in L^\infty\left(0,T;L^{\frac{p}{p-1}}_{loc}(\Omega)\right)$ and $r^{\eps}$ converges to zero in $L^1L^p_{loc}$, the right hand side vanishes in the limit. 
We deduce
\[
\int_0^{T_K}\int_\Omega m\eta\diff x\diff t\leq 0,
\]
thus proving the assertion of the lemma as $\eta\geq0$ was taken arbitrarily from $C^\infty_c((0,T_K)\times U)$.

\end{proof}

Finally, we are in the position to prove the main result, Theorem~\ref{mainthm}, which claims $\omega^\nu\to\omega$ in $C_{loc}([0,T];L^q_{loc}(\Omega))$ for any $1\leq q < p$.

\begin{proof}[Proof of Theorem~\ref{mainthm}]
We fix an arbitrary compact subset $K$ of $\Omega$. 

{\em Step 1.} By lower semi-continuity and Lemma \ref{L3} with $q=p-1$, it is clear that
\[
0 \le \liminf_{\nu\to 0} \int_0^T\int_\Omega \left(|\omega^{\nu}|^{{p-1}}-|\omega|^{{p-1}}\right)\psi \, \diff x \diff t = \int_0^T\int_\Omega m \psi\, \diff x \diff t \le 0,
\]
for any non-negative $\psi \in C^{\infty}_c((0,T_K)\times K)$, and thus, $m=0$ in $(0,T_K)\times K$. Choosing $T_K$ as the initial time, the argument can be repeated (recall that $T_K$ from Lemma~\ref{L3} depended only on $K$ and the supremum norm of $u$!) and we obtain that $m=0$ in $(0,T)\times K$.
Since weak convergence and convergence of norms together imply strong convergence, it follows that $\omega^{\nu}$ is strongly convergent to $\omega$ in $L^q((0,T)\times K)$ for any $q\le {p-1}$ and then by interpolation also in $[1,p)$. 

{\em Step 2.} It remains to show that this convergence is in fact uniform in time. Let again $\chi\in C_c^\infty(\Omega)$ with support in $K$. First, from~\eqref{renormlimit} we see that $\frac{\diff }{\diff t}\int_{\Omega}|\chi\omega(t)|^q\diff x\in L^\infty(0,T)$ in the sense of distributions. Similarly, \eqref{renormnu} reveals $\frac{\diff }{\diff t}\int_{\Omega}|\chi\omega^\nu|^q\diff x\leq C$, where the constant $C$ is independent of the viscosity $\nu$. Recalling Corollary~\ref{ctycor} and the local $L^q$ convergence of the initial vorticities, we find that the assumptions of Lemma~\ref{uniformlem} are satisfied for
\begin{equation*}
    f_\nu(t):=\int_\Omega |\chi\omega^\nu(t)|^q\diff x,\quad\quad f(t):=\int_\Omega |\chi\omega(t)|^q\diff x
\end{equation*}
(the convergence $f_{\nu}\to f$ in $L^1(0,T)$, is precisely what has been established in Step 1). We conclude $f_n\to f$ uniformly on $[0,T']$ for any $T'<T$, and the result of this proof step is 
\begin{equation}\label{convergenceofnorms}
    \int_\Omega |\chi\omega^\nu(t)|^q\diff x\to\int_\Omega |\chi\omega(t)|^q\diff x\quad \text{uniformly in $[0,T']$, for any $1\leq q < p$.}
\end{equation}

{\em Step 3.} For $\chi$ as above, we want to prove $\chi\omega^\nu\to\chi\omega$ in $C([0,T'];L^q(\Omega))$. Assume this were not the case, so there existed a $\delta>0$ and, for each $\nu>0$, a time $t^\nu\in[0,T']$ such that
\begin{equation}\label{contradiction}
    \|\chi\omega^\nu(t^\nu)-\chi\omega(t^\nu)\|_{L^q}\geq \delta.
\end{equation}
Up to a subsequence, $t^\nu\to t^*\in[0,T]$. Then, for any $\psi\in L^{q'}(\Omega)$,
\begin{equation*}
    \left|\int_\Omega \psi\chi(\omega^\nu(t^\nu)-\omega(t^*))\diff x\right|\leq \left|\int_\Omega \psi\chi(\omega^\nu(t^\nu)-\omega(t^\nu))\diff x\right|+\|\psi\|_{L^{q'}}\|\chi(\omega(t^\nu)-\omega(t^*))\|_{L^q}\to0 
\end{equation*}
as $\nu\to0$, because the first term vanishes due to $\chi\omega^\nu\to \chi\omega$ in $C([0,T'];L^q_w(\Omega))$ (Lemma~\ref{L1}), and so does the second term by Corollary~\ref{ctycor}. This means $\chi\omega^\nu(t^\nu)\rightharpoonup\chi\omega(t^*)$ weakly in $L^q(\Omega)$.
Therefore, by weak lower semicontinuity,
\begin{align*}
    \|\chi\omega(t^*)\|_{L^q}&\leq \liminf_{\nu\to0}\|\chi\omega^\nu(t^\nu)\|_{L^q}\leq \limsup_{\nu\to0}\|\chi\omega^\nu(t^\nu)\|_{L^q} \\
    & \leq \limsup_{\nu\to0}\left[\left(\|\chi\omega^\nu(t^\nu)\|_{L^q}-\|\chi\omega(t^\nu)\|_{L^q}\right)+\|\chi\omega(t^\nu)\|_{L^q}\right]\\
    & \leq \limsup_{\nu\to0}\left[\sup_{t\in[0,T']}\left(\|\chi\omega^\nu(t)\|_{L^q}-\|\chi\omega(t)\|_{L^q}\right)+\|\chi\omega(t^\nu)\|_{L^q}\right]\\
    &=\|\chi\omega(t^*)\|_{L^q},
\end{align*}
where for the final estimate we used~\eqref{convergenceofnorms} and, once more, Corollary~\ref{ctycor}. In total, we arrive at $\|\chi\omega(t^*)-\chi\omega^\nu(t^\nu)\|_{L^q}\to0$. Keeping in mind $\chi\omega(t^\nu)\to\chi\omega(t^*)$ in $L^q(\Omega)$, this gives us the desired contradiction with~\eqref{contradiction}.

\medskip

\emph{Step 4}. We finally establish the vanishing enstrophy dissipation by testing the \eqref{sys:navier_stokes_vorticity_formulation} equations with $\chi \omega^\nu$ and integrating by parts,
\begin{align*}
\MoveEqLeft \frac12 \int_{\Omega} \chi (\omega^{\nu}(T',\cdot))^2\, \diff x - \frac12 \int_{\Omega} \chi (\omega^{\nu}_0)^2\, dx +\nu \int_0^{T'} \int_{\Omega} \chi |\nabla \omega^{\nu}|^2\, \diff x \diff t \\
& = \frac{\nu}2 \int_0^{T'}\int_{\Omega} \Delta \chi (\omega^{\nu})^2\, \diff x \diff t +\frac12\int_0^{T'}\int_{\Omega} \nabla\chi \cdot u^{\nu}(\omega^{\nu})^2\, \diff x \diff t.
\end{align*}
Passing to the limit $\nu\to 0$ and using the established strong convergence results, which actually entail that the velocity fields are uniformly converging, we deduce that
\begin{align*}
\MoveEqLeft \frac12 \int_{\Omega} \chi (\omega(T',\cdot))^2\, \diff x - \frac12 \int_{\Omega} \chi \omega_0^2\, \diff x +\lim_{\nu\to 0} \nu \int_0^{T'} \int_{\Omega} \chi |\nabla \omega^{\nu}|^2\, \diff x \diff t \\
& = \frac12\int_0^{T'}\int_{\Omega} \nabla\chi \cdot u\, \omega^2\, \diff x \diff t.
\end{align*}
The last term on the left-hand side must be zero because $\omega$ is a renormalized solution.
\end{proof}

\appendix
\section{Renormalization in Lemma \ref{L1}}\label{proof_renormalization}

 First, we need to extend our equation to negative times in order to keep the initial conditions. Let $\overline{u}$ be the extension of $u$ by $0$ and
    \begin{align*}
    \overline{\omega}(t,x) = \begin{cases}
    \omega(t, x) &\text{ when }t\in (0, T],\\
    \omega_0(x) &\text{ when }t \leq 0.
    \end{cases}
    \end{align*}
    Fix an arbitrary $\varphi\in C^\infty_c((-T, T)\times \Omega)$. Then, by \eqref{weak_formulation_vorticity},
    \begin{equation}\label{eq:weak_formulation_vorticity_negative_times}
        \begin{split}
            \int_{-T}^T\int_{\Omega}\overline{\omega}\,\p_t\varphi + &(\overline{u}\cdot\nabla\varphi)\,\overline{\omega}\diff x\diff t = \\
            &= \int_{-T}^0\int_{\Omega}\omega_0(x)\,\p_t\varphi(t,x)\diff x\diff t + \int_0^T\int_{\Omega}\omega\,\p_t\varphi\diff x\diff t + \int_0^T\int_{\Omega}(u\cdot\nabla\varphi)\,\omega\diff x\diff t\\
            &= \int_{\Omega}\omega_0\varphi(0, x)\diff x + \int_0^T\int_{\Omega}\omega\,\p_t\varphi\diff x\diff t + \int_0^T\int_{\Omega}(u\cdot\nabla\varphi)\,\omega\diff x\diff t  = 0.
        \end{split}
    \end{equation}
    Choosing in \eqref{eq:weak_formulation_vorticity_negative_times}  a space-mollified test function $\varphi^{\eps}=\eta_{\eps}\ast \varphi$ (which we can do for small mollification parameters $\eps$ because $\varphi$ is compactly supported in $\Omega$), we see that
    \begin{align*}
        \int_{-T}^T\int_{\Omega}\overline{\omega}^\varepsilon\,\p_t\varphi\diff x\diff t =  \int_{-T}^T\int_{\Omega}\DIV(\overline{u}\,\overline{\omega})^\varepsilon\,\varphi\diff x\diff t,
    \end{align*}
    which together with Lemma~\ref{cor:bounds_on_velocity_by_sobolev_embed} implies
    $$
    \overline{\omega}^\varepsilon \in W^{1,1}((-T, T)\times \Omega')
    $$
    for any subdomain $\overline{\Omega'}\subset\Omega$. In particular
    \begin{align}\label{eq:strong_formulation_for_vorticity_molified}
        \p_t\overline{\omega}^\varepsilon + \overline{u}\cdot\nabla\overline{\omega}^\varepsilon = r_{\varepsilon},
    \end{align}
    where
    $$
    r_{\varepsilon} = \overline{u}\cdot\nabla\overline{\omega}^\varepsilon-\DIV(\overline{u}\,\overline{\omega})^\varepsilon  .
    $$
    From the Commutator Lemma \ref{commutator_lemma} and the known regularity  $\overline{u}\in L^\infty((-T, T); W^{1, p}_{loc}(\Omega))$ and  $\overline{\omega}\in L^\infty((-T, T); L^p_{loc}(\Omega))$ (recall $p>2$) we can deduce
    \begin{align}
        r_\varepsilon \rightarrow 0 \text{ as }\varepsilon\to 0^+,\text{ strongly in }L^1\left((-T, T); L^{\frac{p}{2}}_{loc}(\Omega)\right).
    \end{align}
    Take a function $\beta\in C^1(\R)$ for which $\beta'$ is bounded in $\R$, and multiply \eqref{eq:strong_formulation_for_vorticity_molified} by $\beta'(\overline{\omega^\varepsilon})$. With the use of the chain rule for Sobolev mappings \cite[Theorem 4.4]{evans2015measure} we get
    \begin{align}\label{eq:renormalized_vorticity_molified}
        \partial_t\beta(\overline{\omega}^\varepsilon) + \overline{u}\cdot\nabla\beta(\overline{\omega}^\varepsilon) = r_{\varepsilon}\beta'(\overline{\omega}^\varepsilon).
    \end{align}
    Fix an arbitrary $\psi\in C^\infty_c([0, T) \times \Omega)$ and extend it smoothly to negative times, in a way that $\overline{\psi}\in C^\infty_c((-T, T)\times\Omega)$, where $\overline{\psi}$ denotes the extension. Multiply \eqref{eq:renormalized_vorticity_molified} by $\overline{\psi}$ and integrate by parts to obtain
    \begin{align*}
        \int_{0}^T\int_{\Omega}\beta(\omega^\varepsilon)\,\p_t\psi + (u\cdot\nabla\psi)\,\beta(\omega^\varepsilon)\diff x\diff t = -\int_{-T}^T\int_{\Omega}r_\varepsilon\,\beta'(\overline{\omega}^\varepsilon)\,\overline{\psi}\diff x\diff t  -\int_{\Omega}\beta(\omega_0^\varepsilon)\psi(0, x)\diff x.
    \end{align*}
    Converging with $\varepsilon\to 0^+$ we get
    \begin{align}\label{eq:renormalized_weak_formulation_vorticity_beta_form}
        \int_{0}^T\int_{\Omega}\beta(\omega)\,\p_t\psi + (u\cdot\nabla\psi)\,\beta(\omega)\diff x\diff t = -\int_{\Omega}\beta(\omega_0)\psi(0, x)\diff x.
    \end{align}
    By approximation, we can pick $\beta(t) = \mathrm{min}(|t|^q, M)$ in \eqref{eq:renormalized_weak_formulation_vorticity_beta_form} to see
    \begin{align}\label{eq:renormalized_weak_formulation_vorticity_minimum_form}
        \int_{0}^T\int_{\Omega}\mathrm{min}(|\omega|^q, M)\,\p_t\psi + (u\cdot\nabla\psi)\,\mathrm{min}(|\omega|^q, M)\diff x\diff t = -\int_{\Omega}\mathrm{min}(|\omega_0|^q, M)\psi(0, x)\diff x,
    \end{align}
    and with the use of Lebesgue's dominated convergence theorem we can converge with $M\to +\infty$ in \eqref{eq:renormalized_weak_formulation_vorticity_minimum_form} to \eqref{eq:renormalized_weak_formulation_vorticity}.

\section{Auxiliary propositions}

\begin{prop}\label{aubin_lions_lemma}\textbf{\textup{(Aubin-Lions lemma~\cite{CJL14})}}\\
    Suppose $X_0, X, X_1$ are Banach spaces such that $X_0$ embeds compactly into $X$ and $X$ embeds continuously into $X_1$. Then 
    $$
    W := \left\{u\in L^p((0, T); X_0),\, \frac{d}{dt}u\in L^q((0, T); X_1)\right\}
    $$
    embeds compactly into
    \begin{itemize}
        \item $L^p((0, T); X)$, for $p < +\infty$ and $1\leq q\leq +\infty$,
        \item $C([0, T]; X)$, for $p = +\infty$ and $1 < q \leq +\infty$.
    \end{itemize}
\end{prop}

\begin{prop}\label{commutator_lemma}\textbf{\textup{(Commutator lemma~\cite[Lemma II.1]{diperna1989ordinary})}}\\
    Let $B\in L^1((0, T); W^{1,\alpha}_{loc}(\R^d))$ and $w\in L^\infty((0, T); L^p_{loc}(\R^d))$, where $1\leq p\leq +\infty$ and $\alpha \geq \frac{p}{p-1}$. Then
    $$
    (B\cdot\nabla_x w)^\varepsilon - B\cdot\nabla_x w^\varepsilon \rightarrow 0 \text{ as }\varepsilon\to 0^+\text{, strongly in }L^1((0, T); L^\beta_{loc}(\R^d)),
    $$
    for $\frac{1}{\beta} = \frac{1}{\alpha} + \frac{1}{p}$.
\end{prop}

\begin{lem}\label{uniformlem}
Let $T>0$ and $(f_n)_{n\in\N}$ a sequence of continuous functions which converge to the continuous function $f$ in $L^1(0,T)$. Assume there is $C\in\R$ such that $f_n'\leq C$ and $f'\leq C$ in the sense of distributions, for each $n\in\N$. Assume moreover $f_n(0)\to f(0)$.

Then, for each $T'<T$, the convergence is uniform in $[0,T']$.
\end{lem}

\begin{proof}
Without loss of generality, we may assume $C=0$ (otherwise consider $f_n-Ct$ and $f-Ct$). In other words, the functions are non-increasing.

    Choosing a subsequence, we may 
    assume $f_n\to f$ almost everywhere. 
    
    Let $\epsilon>0$ and $\delta>0$ so small that $|f(t)-f(t')|<\frac\epsilon2$ whenever $|t-t'|<\delta$. Pick a finite partition $(t_k)\subset[0,T]$ of fineness at most $\delta$ and such that $f_n(t_k)\to f(t_k)$ as $n\to\infty$, for each $k$. We may take $t_0=0$ (by assumption) and assume that at least one $t_k$ lies in $(T',T)$.   
    
    Let $t\in[0,T']$ and pick $t_k\leq t$ such that $f(t_k)-f(t)<\frac{\epsilon}{2}$. Then,
    \begin{equation*}
    f_n(t)-f(t)\leq f_n(t_k)-f(t_k)+\frac{\epsilon}{2}<\epsilon 
    \end{equation*}
    for sufficiently large $n$. Analogously, choosing $t_k\geq t$, we have $f(t)-f_n(t)<\epsilon$ for sufficiently large $n$. As the attribute `sufficiently large' pertains only to the finitely many $t_k$, we see that $n$ can be chosen independently of $t$, and we obtain uniform convergence in $[0,T']$. 

    As every uniformly convergent subsequence has the same limit $f$, in fact the whole sequence converges uniformly.
\end{proof}

\bibliographystyle{plain}
\bibliography{2deuler}
\end{document}